\documentclass[11pt]{amsart}
\usepackage{pdfsync}

\usepackage{amsmath,amssymb,amscd,amsfonts,verbatim}
\usepackage{paralist}
\usepackage[mathscr]{eucal}
\usepackage{comment}
\usepackage{xypic}
\usepackage{txfonts}

\usepackage{enumitem}
\usepackage{enumerate}

\usepackage[pdftex]{hyperref}
\hypersetup{citecolor=blue,linktocpage}

\newtheorem{thm}{Theorem}[section]
\newtheorem{lem}[thm]{Lemma}

\newtheorem{prop}[thm]{Proposition}
\newtheorem{cor}[thm]{Corollary}

\newtheorem{assu-nota}[thm]{Assumption--Notation}
\theoremstyle{definition}

\newtheorem{rem}[thm]{Remark}

\newcommand{\C}{\mathbb C}

\newcommand{\Z}{\mathbb Z}

\newcommand{\F}{\mathbb F}
\newcommand{\G}{\mathbb G}

\newcommand{\pp}{\mathbb P}
\newcommand{\Oc}{\mathcal O}

\newcommand{\Fc}{\mathcal F}
\newcommand{\Ec}{\mathcal E}

\newcommand{\Mm}{\mathcal M}

\DeclareMathOperator{\Pic}{Pic}

\usepackage[colorinlistoftodos]{todonotes}

\def\geq{\geqslant}
\def\leq{\leqslant}
\numberwithin{equation}{section}
\title[Birational geometry]{Birational geometry of the twofold symmetric product of a Hirzebruch surface via secant maps}
\author[Marco Andreatta, Ciro Ciliberto and Roberto Pignatelli]{Marco Andreatta, Ciro Ciliberto and Roberto Pignatelli} 
\address[Marco Andreatta] {Dipartimento di Matematica, Universit\`a di Trento, via Sommarive 14, 38123 Trento, Italy}
\email{marco.andreatta@unitn.it}
\address[Ciro Ciliberto]{Dipartimento di Matematica, Universit\`a di Roma Tor Vergata, via della Ricerca Scientifica, 00173 Roma, Italy}
\email{cilibert@mat.uniroma2.it}
\address[Roberto Pignatelli]{Dipartimento di Matematica, Universit\`a di Trento, via Sommarive 14, 38123 Trento, Italy}
\email{roberto.pignatelli@unitn.it}

\thanks{All authors are members of the GNSAGA group of INDAM. The first and the third author were partially supported by a grant of MUR: PRIN 2017 {\it Moduli Theory and Birational Classification}. The third author was partially supported  by the European Union - Next Generation EU, Mission 4 Component 2 - CUP E53D23005400001.
}

\begin{document}
\begin{abstract} In this paper, extending some ideas of Fano in \cite{Fa} and of the first and last author in \cite{AP}, we study the birational geometry of the Hilbert scheme of  0--dimensional subschemes of length 2 of a rational normal scroll $\F_n$.  This fourfold has three elementary contractions associated to the three faces of its nef cone. We study natural projective realizations of these contractions. In particular, given a smooth rational normal scroll $S_{a,b}$ of degree $r$ in $\pp^{r+1}$ with $1\leq a\leq b$ and $a+b=r$, i.e., 
$S_{a,b}=\pp(\Oc_{\pp^1}(a)\oplus \Oc_{\pp^1}(b))$ embedded in $\pp^{r+1}$ with its $\Oc(1)$ line bundle (from an abstract viewpoint $S_{a,b}\cong \F_{b-a}$), we consider the variety $X_{a,b}\subset \G(1,r+1)$ described by all lines that are secant or tangent to $S_{a,b}$. The variety $X_{a,b}$ is the image of some of the aforementioned contractions, it is smooth if $a>1$, and it is singular at a unique point if $a=1$. We compute the degree of $X_{a,b}$ and the local structure of the singularity of $X_{a,b}$ when $a=1$. Finally we discuss in some detail the case $r=4$, originally considered by Fano in \cite{Fa}, because the smooth hyperplane sections of $X_{2,2}$ and $X_{1,3}$ are the Fano 3--folds that appear as number 16 in the Mori--Mukai list of Fano 3--folds with Picard number 2. We prove that any smooth hyperplane section of 
$X_{2,2}$ is also a hyperplane section of $X_{1,3}$, and we discuss the GIT--stability of the smooth hyperplane sections of  $X_{1,3}$ where  $G$ is the subgroup of the projective automorphisms of $X_{1,3}$ coming from the ones of $S_{1,3}$.
\end{abstract}
\maketitle

\section{Introduction}\label{sec:0}
In 1949 Fano published his last paper on $3$-folds where he constructed a smooth  $3$-fold of degree $22$ in a projective space of dimension $13$ with canonical curve section, \cite{Fa}. The first and last authors recently revised in  \cite{AP} this paper with the dual purpose of providing  a detailed proof of all Fano's claims and  of setting up the construction in modern language. This $3$-fold is in fact a Fano manifold in modern sense, that is its anticanonical divisor is very ample, and it was neglected by next mathematicians. In \cite{AP} it was denoted as  {\it Fano's last Fano} (FlF)  and it was pointed out that it is the number 16 in the Mori--Mukai list of Fano 3--folds with Picard number 2 (see \cite{MM}).

Fano obtained his FlF has an hyperplane section of the $4$-fold contained in the Grassmannian of lines in $\pp^5$, which is the union of all secants (and tangents) lines of a general rational normal scroll of degree four in $\pp^5$. His ingenuous geometric construction fits particularly well with the algebraic concept of Hilbert schemes of points on a surface, developed few years later, and could be subject of many generalizations.  In this  paper we propose some of them.

\smallskip
Consider the Hilbert scheme of length two 0--dimensional subschemes of the Hirzebruch surfaces $\pi: \F_n\longrightarrow \pp^1$, which we denote by $\F_n[2]$. This abstract scheme was studied primarily by Fogarty, who proved that it is a smooth variety of dimension $4$ and Picard rank $3$, \cite{Fog}. 

Subsequently it was noted, see \cite{BC}, that $\F_n[2]$ has three elementary contractions associated to the three faces of its nef cone. One is the map $\phi_n: \F_n[2] \longrightarrow \F_n(2)$, which is the obvious map from the Hilbert scheme to the Chow scheme, or symmetric product. The other two, which we denote by $\phi_{n,i}: \F_n[2] \longrightarrow Z_{n,i}$ for $i=1,2$, are respectively described as follows.  The map $\phi_{n,1}$ contracts the divisor of all pairs of points on a fiber of $\pi$ to a  smooth, rational curve $\Gamma\subset Z_{n,1}$; one can see, Proposition (\ref{prop:z1}),  that $Z_{n,1}$ is smooth and $\phi_{n,1}$ is the blow--up of $\Gamma$. For $n\geq 1$, $\phi_{n,2}$ contracts the surface $\Ec_n\cong \pp^2$ of all pairs of points on $E$, the section of $\pi$ with self intersection $-n$, to a point $\mathfrak p$ . The map $\phi_{0,2}$ on the other side is as $\phi_{0,1}$, i.e.,  it is a smooth blow-down obtained contracting all  pairs of points on a fiber of the other ruling of  $\F_0$. 

Note that the exceptional loci of $\phi_{n,1}$ and $\phi_{n,2}$ are disjoint, so the two contractions can be performed independently obtaining a morphism $\psi_n: \F_n[2] \longrightarrow X_n$ into a variety $X_n$ which is smooth for $n=0$, whereas for $n>1$  it has a singular point in $\mathfrak p$.

To obtain a {\it projective realization} of the above varieties and morphisms we consider the smooth rational normal scrolls $S_{a,b}$ of degree $r$ in $\pp^{r+1}$. More precisely $S_{a,b}=\pp(\Oc_{\pp^1}(a)\oplus \Oc_{\pp^1}(b))$, 
with $1\leq a\leq b$ and $a+b=r$, embedded in $\pp^{r+1}$ with its $\Oc(1)$ line bundle. From an abstract viewpoint one has $S_{a,b}\cong \F_{b-a}$. 

Take the morphism
$$
\gamma_{a,b}: S_{a,b}[2]\cong \F_{b-a}[2] \longrightarrow \G(1,r+1)\subset \pp^{\frac {r(r+3)}2}
$$
to the Grassmannian of lines in $\pp^{r+1}$ in its Pl\"ucker embedding, acting in the following way: each 0--dimensional length 2 scheme $\mathfrak y\in S_{a,b}[2]$ is mapped by $\gamma_{a,b}$ to the line $\ell_\mathfrak y:=\langle \mathfrak y\rangle$ spanned in $\pp^{r+1}$ by $\mathfrak y$. 
We call $\gamma_{a,b}$  the \emph{secant map} of $S_{a,b}$ and we denote its image by $X_{a,b}$. 

The case $a=b=1$ is trivial, since $X_{1,1}=\mathbb G(1,3)$. On the other hand $X_{2,2}$ is smooth and its is the $4$-fold considered by Fano, whose smooth hyperplane sections are FlFs. 

The secant map 
$$\gamma_{a,b}:S_{a,b}[2]\longrightarrow X_{a,b}$$  
is a projective realization of the map $\phi_{b-a,1}: \F_{b-a}[2]\longrightarrow Z_{b-a,1}$ if $a\geq 2$. If $a=1$ it is projective realization of the map $\psi_{b-1}: \F_{b-1}[2]\longrightarrow X_{b-1}$.

\smallskip
In the first part of the paper we tackle two questions.  First we compute the degree of  $X_{a,b}$ in the Pl\"ucker embedding in $\pp^{\frac {r(r+3)}2}$; 
namely we prove that $\deg (X_{a,b})= 3r^2-8r+6$ (see Theorem (\ref{thm:deg})). 
The formula agrees with the computation in \cite{Cattaneo} of the degree of the image of the secant map for a general surface $S$; however the assumption in \cite{Cattaneo} is  quite  stronger, namely the surface $S$ has to be embedded by a $3$-very ample line bundle.

Then we prove that  $X_{a,b}$ is smooth, except for $X_{1,r-1}$, in which case  it has a unique singular point at $\mathfrak p$, whose tangent cone is the cone over a smooth threefold linear section of the Segre embedding of  $\pp^2\times \pp^{r-2}$ in $\pp^{3r-4}$ (see Theorem (\ref{thm:sing})).

In Section $4$ we discuss the case $r=4$ which gives rise to the FlF. The variety $X_{2,2}$ is, as we said, smooth while $X_{1,3}$ has a singular point. They are both Fano in Fano's sense, that is their general curve section is a smooth and canonical of genus $12$. Their general hyperplane section is a FlF and moreover any FLF is a hyperplane  section of a $X_{1,3}$ (Remark(\ref{rem:eq}). 

The surface $S_{1,3}$ is a specialization of $S_{2,2}$ and therefore $X_{1,3}$ is a specialization of $X_{2,2}$. Most of the section is devoted to describe in different ways how this specialization takes place. In particular in Remark(\ref{rem:qquad}) we give a nice geometric description of it via a series of birational maps, starting from the specialization of the smooth quadric $S_{1,1}$ to the singular quadric $S_{0,2}$. We use the facts that $X_{2,2}$ can be obtained as the blown-up of  $X_{1,1} =\mathbb G(1,2)$ along a conic not contained in a plane of $\mathbb G(1,3)$. 
These examples could be useful to test some conjectures on K-stability of deformations of Fano Varieties.

Finally in Section $5$ we approach the question of defining an appropriate moduli space for the FlFs, studying when a hyperplane section of $X_{1,3}$ is G-stable, where $G$ is the subgroup of the projective automorphisms of $X_{1,3}$ coming from the ones of $S_{1,3}$.

\section{Hirzebruch surfaces and their twofold symmetric product}\label{sec:1}

In this section we will consider Hirzebruch surfaces $\F_n$, with their structure morphism $\pi: \F_n\to \pp^1$. We denote by $F$ a fibre of $\pi$ and by $E$ a section of $\pi$ such that $E^2=-n$ (this section is unique if $n>0$). We will abuse notation and denote by $F$ and $E$ also the line bundles on $\F_n$ associated to $F$ and $E$. We will keep this notation throughout the paper.

We denote by $\F_n[2]$ the Hilbert scheme of length two 0--dimensional subschemes of $\F_n$, that is a smooth variety of dimension 4. We also have the symmetric product $\F_n(2)=	\F_n\times \F_n/\Z_2$ and the obvious morphism
$ f_n: \F_n[2]\longrightarrow \F_n(2)$.  

Let $L$ be a line bundle on $\F_n$. Then $L^{\boxtimes 2}:=L\boxtimes L$ is a line bundle on $\F_n\times \F_n$ invariant under the action of $\Z_2$, so $L^{\boxtimes 2}$ descends to a line bundle $L(2)$ on $\F_n(2)$. We denote by $L[2]$ the line bundle $f_n^*(L(2))$ on $ \F_n[2]$. In particular we can consider the line bundles $F[2]$ and $E[2]$ on $\F_n[2]$. On $\F_n[2]$ there is a third relevant divisor (or line bundle), namely the \emph{diagonal}, i.e., the set $B$ of all non--reduced subschemes in $\F_n[2]$, that is contracted to the $2$--dimensional diagonal of $\F_n(2)$ by $f_n$. It turns out that $B$ is divisible by $2$ in $\Pic(\F_n[2])$ so that we can consider the line bundle $\frac B2$. 
We note that $F[2]$ and $E[2]$ and $\frac B2$ are a basis of $\Pic(\F_n[2])$ (see \cite{Fog}).

The following result is contained in \cite [Thm. 1]{BC}:

\begin{thm}\label{thm:BC} The nef cone of $\F_n[2]$ is the convex hull of the rays generated by  
$$
F[2], \quad E[2]+nF[2], \quad E[2]+(n+1)F[2]-\frac B2.
$$
\end{thm}

The three faces of this cone correspond to contractions of extremal rays (i.e., the relative Picard group is $\Z$), and precisely the contractions in question are described in the following diagram
$$
 \xymatrix{ Z_{n,1}& \ar[l]_{\phi_{n,1}} \F_n[2] \ar[r]^{\phi_{n,2}}  \ar[d]_{\phi_n} & Z_{n,2} \\
      &\F_n(2) } 
$$
 where:\\
 \begin{inparaenum}
 \item [(i)] $\phi_n=f_n$ is the map determined by the face bounded by $F[2]$ and  $E[2]+nF[2]$;\\
 \item [(ii)] $\phi_{n,1}$  is the map determined by the face bounded by $F[2]$ and  $E[2]+(n+1)F[2]-\frac B2$;\\
 \item [(iii)] $\phi_{n,2}$  is the map determined by the face bounded by $E[2]+nF[2]$ and  $E[2]+(n+1)F[2]-\frac B2$.
\end{inparaenum}

Let $F(2)$ be the curve in $\F_n[2]$ described by the pairs of points of a $g^1_2$ on a curve $F$. One has $F(2)\cdot F[2]=F(2)\cdot (E[2]+(n+1)F[2]-\frac B2)=0$ (see \cite [p. 23]{BC}), hence $\phi_{n,1}$ contracts the divisor $\Fc_n$ on $\F_n[2]$ described by all pairs of points on a curve in $|F|$, to a curve $\Gamma\subset Z_{n,1}$ isomorphic to $\pp^1$. If $n=0$, similarly  $\phi_{0,2}$ contracts the divisor $\Fc'_0$ on $\F_0[2]$ described by all pairs of points on a curve in $|E|$, to a curve $\Gamma'\subset Z_{0,2}$ isomorphic to $\pp^1$. 
If $n>0$, let $E(2)$ be the curve in $E[2]$ described by all pairs of points of a $g^1_2$ instead. Then
$
E(2) \cdot \left( E[2]+nF[2] \right) = E(2) \cdot \left( E[2] + \left( n+1 \right) F[2] -\frac{B}2 \right) =0
$
(see again  \cite [p. 23]{BC}), hence  $\phi_{n,2}$ contracts to a point $\mathfrak p$ the surface $\Ec_n\cong \pp^2$ of all pairs of points on $E$. 

\begin{prop}\label{prop:z1} The 4--dimensional variety $Z_{n,1}$ is smooth and $\F_n[2]$ is the blow--up of $Z_{n,1}$ along the smooth rational curve $\Gamma$. If $n=0$, the same  happens for  $Z_{0,2}$.
\end{prop}

\begin{proof} One has $E[2]\cdot F(2)=1 $ (see \cite [p. 23]{BC}), hence $\phi_{n,1}$ is the contraction of the ray $R$ generated by $F(2)$. Now notice that 
$$
K_{\F_n[2]}\equiv K_{\F_n}[2]=-2E[2]-(n+2)F[2]
$$
(see \cite [Proof of Cor. 3]{BC}). Hence
$$
\ell(R):=\inf\{-K_{\F_n[2]}\cdot C| C\in R\}=2.
$$
Since $\ell(R)$ equals the dimension of the fibres of $\phi_{n,1}$, we can apply \cite [Thm. 5.1]{AO}, which implies the first assertion. The final assertion is clear. \end{proof}

If $n>1$, since the morphism $\phi_{n,2}$ is a small (hence crepant) contraction, the variety $Z_{n,2}$ cannot be smooth at $\mathfrak p$, whereas it is smooth off $\mathfrak p$. We do not dwell now on the nature of the singularity of $Z_{n,2}$ at $\mathfrak p$, but we will return on this later. For the time being we notice that the loci $\Fc_0$ and $\Fc'_0$, and the loci
$\Fc_n$ and $\Ec_n$ for $n>1$, contracted by $\phi_{n,1}$ and $\phi_{n,2}$ are disjoint. So the contractions of $\Fc_n$ and $\Ec_n$ can be performed independently obtaining a morphism $\psi_n$ appearing in the following commutative diagram
\begin{equation}\label{eq:diag}
 \xymatrix{  \F_n[2] \ar[r]^{\phi_{n,1}}  \ar[d]_{\phi_{n,2}}\ar[dr]^{\psi_n} & Z_{n,1}\ar[d]^{\psi_{n,1}} \\
      Z_{n,2}\ar[r]_{\psi_{n,2}}  & X_n } 
      \end{equation}
The variety $X_n$ is smooth for $n=0$, whereas for $n>1$  it is singular at the point $\mathfrak q=\psi_{n,2}(\mathfrak p)$, where $X_n$ has the same singularity as $Z_{n,2}$.

\section{The secant map}

Let us consider now a linearly normal smooth rational normal scroll $S_{a,b}$ of degree $r$ in $\pp^{r+1}$, where
$$
S_{a,b}=\pp(\Oc_{\pp^1}(a)\oplus \Oc_{\pp^1}(b))
$$
with $1\leq a\leq b$ and $a+b=r$, embedded in $\pp^{r+1}$ with its $\Oc(1)$ line bundle. From an abstract viewpoint one has $S_{a,b}\cong \F_{b-a}$. We will consider the morphism
$$
\gamma_{a,b}: S_{a,b}[2]\cong \F_{b-a}[2] \longrightarrow \G(1,r+1)\subset \pp^{\frac {r(r+3)}2}
$$
to the Grassmannian of lines in $\pp^{r+1}$ in its Pl\"ucker embedding, acting in the following way: each 0--dimensional lenght 2 scheme $\mathfrak y\in S_{a,b}[2]$ is mapped by $\gamma_{a,b}$ to the line $\ell_\mathfrak y:=\langle \mathfrak y\rangle$ spanned in $\pp^{r+1}$ by $\mathfrak y$. We call $\gamma_{a,b}$  the \emph{secant map} of $S_{a,b}$. Its image is denoted by $X_{a,b}$. 

\begin{lem}\label{sec}  Let $Z \subset \pp^r$ be a closed subscheme whose ideal is generated by quadrics. 

Ler $\mathfrak y$ be a length $2$ subscheme of $Z$ and let $l_{\mathfrak y}$ the line $ \left\langle \mathfrak y \right\rangle$ in $\pp^r$. Then the intersection of  $l_{\mathfrak y}$  with $Z$ consists of $\mathfrak y$ unless $l_{\mathfrak y}$  is contained in $Z$. 
In particular given two distinct subschemes $\mathfrak y,\mathfrak z$ of length $2$ in $Z$, $l_{\mathfrak y}=l_{\mathfrak z}$ if and only if there is a line $l$ contained in $Z$ such that both ${\mathfrak y}$ and ${\mathfrak z}$ are subschemes of $l$.
\end{lem}

\begin{proof}  This is an immediate consequence of the fact that the ideal of $Z$ is generated by quadrics. \end{proof}

Note that lemma \ref{sec} applies to $S_{a,b}$ whose ideal is generated by quadrics.

Now we have two main questions to deal with:\\
\begin{inparaenum}
\item [(a)] what is the degree of  $X_{a,b}$ in the Pl\"ucker embedding in $\pp^{\frac {r(r+3)}2}$?\\
\item [(b)] what are the smooth points and the singular points of $X_{a,b}$, if any?
\end{inparaenum}

\subsection{The degree of \texorpdfstring{$X_{a,b}$}{X{a,b}}}
To compute the degree of  $X_{a,b}$ we extend a beautiful argument by Fano (see \cite {Fa}), revisited  in \cite {AP}. 

We compute the degree of the 4--dimensional variety $X_{a,b}$  as the degree of a surface obtained cutting $X_{a,b}$ with  two independent hyperplanes  in $\pp^{\frac {r(r+3)}2}$. We will consider two special hyperplane sections of $\G(1,r+1)$ given by the lines that intersect two linear subspaces $\sigma_1, \sigma_2$ of codimension 2 in  $\pp^{r+1}$ that are in special position, i.e., they span a general hyperplane $\sigma$ of $\pp^{r+1}$ and are general there, and therefore intersect in a general subspace $\pi$ of codimension 3 in $\pp^{r+1}$. 

Then the set of lines in the Grassmanian that intersect both $\sigma_1, \sigma_2$ breaks up in two codimension 2 Schubert cycles: the lines contained in the hyperplane $\sigma$ and the lines intersecting the subspace $\pi$. We denote by $X_{a,b}^\sigma$ the subvariety of $X_{a,b}$ consisting of the lines contained in $\sigma$ and by $X_{a,b}^\pi$ the subvariety of $X_{a,b}$ consisting of the lines intersecting $\pi$.

\begin{lem}\label{lem:int} One has
$$
\deg (X_{a,b})=\deg (X_{a,b}^\sigma)+\deg (X_{a,b}^\pi).
$$
\end{lem}

\begin{proof} The assertion will follow once we prove that $X_{a,b}^\sigma\cup X_{a,b}^\pi$, with its reduced structure, is the scheme theoretical intersection of $X_{a,b}$ with the aforementioned two hyperplanes. To prove this, we will make a computation in coordinates and we will imitate a similar computation made in \cite {AP}. 

First of all we prove that for a general choice of $\sigma$ and $\pi$, $X_{a,b}^\sigma$ and $X_{a,b}^\pi$ are both irreducible.

The assertion is trivial for $X_{a,b}^\sigma$ that is the secant variety of the rational curve intersection of $X_{a,b}$ with $\sigma$. 

As for $X_{a,b}^\pi$, we consider the linear projection $f \colon S_{a,b} \rightarrow \pp^2$ with center  $\pi$. This is a degree $r$ finite cover whose monodromy is the full symmetric group (\cite{pirolaschl, zariski}). 
Then a dense open subset of $X_{a,b}^\pi$ can be identified with the pairs of points contained in any fibre of $f$ consisting of $r$ distinct points. Since the symmetric group is $2-$transitive this open subset is irreducible.

Next we introduce homogeneous coordinates $[x_0,\ldots, x_{r+1}]$ in $\pp^{r+1}$ so that $S_{a,b}$ has equations
\begin{equation*}\label{eq:scroll}
{\rm rk}
 \left(
\begin{array}{cccccc}
x_0&\ldots &x_{b-1}&x_{b+1}&\ldots&x_r  \\
x_{1}&\ldots &x_b&x_{b+2}&\ldots&x_{r+1}
\end{array}
\right)<2.
\end{equation*}
Consider the line $\ell$ with equations $x_{1}=\cdots =x_b=x_{b+2}=\cdots=x_{r+1}=0$, that is a line of the ruling of $S_{a,b}$. An open  neighborhood of $\ell$ in $\G(1,r+1)$ consists of all lines joining the points whose homogeneous coordinates are given by the rows of the following matrix
$$
 \left(
\begin{array}{ccccccccc}
1&\xi_1&\ldots &\xi_{b}&0&\xi_{b+2}&\ldots&\xi_{r+1}  \\
0&\eta_1&\ldots &\eta_b&1&\eta_{b+2}&\ldots&\eta_{r+1}
\end{array}
\right).
$$
This tells us that 
$$
\xi_1, \ldots, \xi_{b}, \xi_{b+2}, \ldots, \xi_{r+1},  \eta_1, \ldots , \eta_b, \eta_{b+2},\ldots,\eta_{r+1}
$$ 
are coordinates of a chart $U$ of $\G(1,r+1)$ centered at $\ell$. A line $t$ parametrized by a point of $U$ has parametric equations of the form
$$
x_0=\lambda, x_{b+1}=\mu, x_i=\lambda \xi_i+\mu \eta_i, \quad i\in \{1,\ldots,r+1\}\setminus \{b+1\}
$$
with $[\lambda, \mu]\in \pp^1$.

The intersection of $t$ with $S_{a,b}$ is obtained by solving in $\lambda, \mu$ the system of equations 
{\tiny
$$
{\rm rk}
 \left(
\begin{array}{ccccccccc}
\lambda&\lambda \xi_{1}+\mu \eta_{1}&\ldots &\lambda \xi_{b-1}+\mu \eta_{b-1} &\mu& \lambda \xi_{b+2}+\mu \eta_{b+2}& \ldots&\lambda \xi_{r}+\mu \eta_{r}   \\
\lambda \xi_{1}+\mu \eta_{1}&\ldots&\ldots &\lambda \xi_{b}+\mu \eta_{b}&\lambda \xi_{b+2}+\mu \eta_{b+2}&\ldots&\ldots&\lambda \xi_{r+1}+\mu \eta_{r+1} 
\end{array}
\right)<2
$$}
that are quadratic in $\lambda,\mu$. 
The line $t$ is secant (or tangent) to $S_{a,b}$ if and only if all the equations in $\lambda, \mu$ that we obtain in this way are proportional. One of these equations, obtained by considering the minor determined by the first and $(b+1)$--th columns  is the following
\begin{equation}\label{eq:prima}
\lambda^2\xi_{b+2}+\lambda\mu (\eta_{b+2}-\xi_1)-\mu^2\eta_1=0
\end{equation}
Suppose that for the line $t$ one has $\eta_1 \neq 0$.

Another equation,  obtained by considering the minor determined by the first and second columns,  is the following
$$
\lambda^2(\xi_2-\xi_1^2)+\lambda\mu(\eta_2-2\xi_1\eta_1)-\mu^2\eta_1^2=0
$$
Since this equation has to be proportional to the one in \eqref {eq:prima}, we get
\begin{equation*}\label{eq:porc}
\xi_2=\xi_1^2+\eta_1\xi_{b+2}, \quad \eta_2=\xi_1\eta_1+\eta_1\eta_{b+2}
\end{equation*}
so that $\xi_2,\eta_2$ can be expressed as polynomials in $\xi_1,\eta_1, \xi_{b+2}, \eta_{b+2}$. Now we claim that also $\xi_i, \eta_i$, with $i\in \{3,\ldots, r+1\}\setminus \{b+2\}$ can be expressed as polynomials in $\xi_1,\eta_1, \xi_{b+2}, \eta_{b+2}$. This can be proved by induction on $i$. Indeed, assume we have proved the assertion for $i$. Then consider the equation 
$$
\lambda (\lambda \xi_{i+1}+\mu \eta_{i+1})=(\lambda\xi_1+\mu\eta_1)(\lambda \xi_{i}+\mu \eta_{i}).
$$
Since this has to be proportional to the one in \eqref {eq:prima}, we get
\begin{equation}\label{eq:plop}
\xi_{i+1}=\xi_1\xi_i+\eta_i\xi_{b+2}, \quad \eta_{i+1}=\xi_i\eta_1+\eta_i\eta_{b+2}
\end{equation}
and, applying induction,  from this we see that also $\xi_{i+1},\eta_{i+1}$ can be expressed as polynomials in $\xi_1,\eta_1, \xi_{b+2}, \eta_{b+2}$. 

Now let $Z\subset X_{a,b}\cap U$ be the set of points with $\eta_1=0$. This is a proper closed subset of $X_{a,b}\cap U$. We notice that $X_{a,b}\cap U \setminus Z$ is contained in $U'$ where $U'$ is the closed embedding of ${\mathbb C}^4$ in $U={\mathbb C}^{2r}$ defined by the equations \eqref{eq:plop} (for $i \in \left\{2 , \ldots , r+1 \right\} \setminus \{ b+2\}$), with variables in ${\mathbb C}^4$ given by $\xi_1, \eta_1, \xi_{b+2}, \eta_{b+2}$. 

Hence  $X_{a,b} \cap U= \overline{\left( X_{a,b} \cap U\right) \setminus Z} \subseteq U'$.
Then, since $X_{a,b}$ is irreducible of dimension $4$, and  $X_{a,b} \cap U$ is closed in $U'$ which is also irreducible of dimension $4$, we deduce that $X_{a,b} \cap U=U' \cong {\mathbb C}^4$ with coordinates $\xi_1, \eta_1, \xi_{b+2}, \eta_{b+2}$.

Now let us consider the two codimension 2 subspaces $\sigma_1, \sigma_2$ with equations
$$
\sigma_1) \quad x_0-x_1=x_{b+1}=0, \qquad \sigma_2) \quad x_0-x_1=x_{b+2}=0
$$
that intersect along the codimension 3 subspace $\pi$ with equations 
\begin{equation}\label{eq:pi}
\pi) \quad x_0-x_1=x_{b+1}=x_{b+2}=0
\end{equation}
and span the hyperplane $\sigma$ with equation  $x_0-x_1=0$. The set of points in the chart $U'$ corresponding to lines intersecting $\sigma_1$ has equation $\xi_1=1$, whereas 
the set of points in $U'$ corresponding to lines intersecting $\sigma_2$ has equation
$(\xi_1-1)\eta_{b+2}=\eta_1\xi_{b+2}$, so the intersection of the two sets has equations
$\xi_1=1, \eta_1\xi_{b+2}=0$, and this splits into two irreducible and reduced components with equations $\xi_1=1, \eta_1=0$ and $\xi_1=1, \xi_{b+2}=0$. The former equations define the set of points in $U'$ belonging to $X_{a,b}^\sigma$, the latter equations define the set of points in $U'$ belonging to $X_{a,b}^\pi$. 

In conclusion, we proved that for a specific choice of $\sigma_1$ and $\sigma_2$ in a hyperplane $\sigma$ and intersecting in a codimension 3 subspace $\pi$  we have that $X_{a,b}^\sigma\cup X_{a,b}^\pi$, with its reduced structure, is the scheme theoretical intersection of $X_{a,b}$ with the two hyperplane sections of lines intersecting $\sigma_1$ and $\sigma_2$. Then this is true for a general choice of $\sigma$, $\sigma_1$ and $\sigma_2$, and  the assertion follows. \end{proof}

To compute the degree of $X_{a,b}$ we have to compute the degrees of $X_{a,b}^\sigma$ and $X_{a,b}^\pi$. Before doing that, we need some preliminary results. 

Let $C_d\subset \pp^d$ be a rational normal curve of degree $d\geq 2$. Let $C_d(2)\cong \pp^2$ be the symmetric product of $C_d$. We have the \emph{secant morphism}
$$
\varphi_d: C_d(2)\cong \pp^2\longrightarrow \G(1,d)\subset \pp^{\frac {d(d+1)}2-1}
$$
 that maps a $\mathfrak y\in C_d(2)$ to the line $\ell_{\mathfrak y}=\langle \mathfrak y\rangle$. We note that $\varphi_d$ is injective by Lemma \ref{sec}. Let us denote by $V_d$ the image of $\varphi_d$. 

\begin{prop}\label{prop:rnc} $V_d$ is the \emph{$(d-1)$--Veronese surface}, i.e., the image of the plane via the complete linear system  $|\Oc_{\pp^2}(d-1)|$, and $\varphi_d$ is an isomorphism onto its image $V_d$.
\end{prop}

\begin{proof} In \cite {AP} the result is proved in the case $d=4$, and the proof of the general case is similar. However we give it here for completeness. 

First of all we prove that $\deg(V_d)=(d-1)^2$. To see this, consider two general linear subspaces $\sigma_1, \sigma_2$ of codimension 2 in $\pp^d$, and intersect $V_d$ with the two hyperplane sections $H_1, H_2$ of $\G(1,d)$ of points corresponding to lines intersecting both $\sigma_1$ and $\sigma_2$. The hyperplanes through $\sigma_i$, for $i=1,2$, cut out on $C_d$ a linear series $g_i$, of dimension 1 and degree $d$. The number of intersection  points of $V_d$ with $H_1\cap H_2$ equals the number of divisors in $C_d(2)$ that are contained at the same time in divisors of $g_1$ and $g_2$. This number is well known to be $(d-1)^2$ (see \cite [p. 344]{ACGH}),  as wanted.

To finish the proof, we have to show that $V_d$ spans $\pp^{\frac {d(d+1)}2-1}$. We prove this by induction of $d$. The assertion is trivial for $d=2$. So, suppose we have proved it for $d$ and let us prove it for $d+1$. Let us consider the rational normal curve $C_{d+1}\subset \pp^{d+1}$ of degree $d+1$ and let $x\in C_{d+1}$ be a point. Let us project down $C_{d+1}$ from $x$, obtaining  a rational normal normal curve $C_d\subset \pp^d$. The projection from $x$ determines also an obvious rational map
$$
p: V_{d+1} \dasharrow V_d
$$
whose indeterminacy locus is a priori contained in the set of secant lines to $C_{d+1}$ that contain $x$, that is the set of lines of the ruling of the cone over $C_{d}$ with vertex $x$. This is a rational normal curve $D$ of degree $d$ on $V_{d+1}$.
Then $p$ is the projection of $V_{d+1}$ to $V_d$ from the subspace spanned by $D$, that has dimension $d$. By induction we know that $V_d$ spans a $\pp^{\frac {d(d+1)}2-1}$. Then $V_{d+1}$ spans a linear space of dimension
$$
\frac {d(d+1)}2-1+ d+1=\frac {d(d+3)}2
$$
as required. 
\end{proof}

\begin{cor}\label{cor:span} $X_{a,b}$ is linearly normal in $\pp^{\frac {r(r+3)}2}$.
\end{cor}
\begin{proof} On $S_{a,b}$ there are rational normal curves $C_{r+1}$ of degree $r+1$, and therefore $V_{r+1}$ is contained in $X_{a,b}$. Since, by Proposition \ref {prop:rnc}, $V_{r+1}$ spans a linear space of dimension $\frac {r(r+3)}2$, we are done. 
\end{proof}

Next we want to figure out how to compute $\deg (X_{a,b}^\pi)$, where $\pi$ is a general linear subspace of codimension 3. To do this we choose a codimension 2 linear subspace $\tau$ that intersects $\pi$ along a linear subspace $\alpha$ of codimension 4, so that $\sigma=\langle \tau, \pi\rangle$ is a hyperplane. Then we intersect $X_{a,b}^\pi$ with the hyperplane section of all lines intersecting $\tau$, thus obtaining a curve $X_{a,b}^{\pi,\tau}$. Then  $X_{a,b}^{\pi,\tau}$ splits in two parts:\\ \begin{inparaenum}
\item [$\bullet$] the curve $X_{a,b}^{\alpha}$ of lines in $X_{a,b}$  intersecting $\alpha$;\\
\item  [$\bullet$] the curve $X_{a,b}^{\sigma,\pi}$ of lines  in $X_{a,b}$ contained in $\sigma$ and intersecting $\pi$. 
\end{inparaenum}

\begin{lem}\label{lem:int2} One has
$$
\deg (X_{a,b}^\pi)=\deg (X_{a,b}^{\alpha})+\deg (X_{a,b}^{\sigma,\pi}).
$$
\end{lem}

\begin{proof} As in Lemma  \ref {lem:int}, the assertion will follow once we prove that $X_{a,b}^{\alpha} \cup X_{a,b}^{\sigma,\pi}$, with its reduced structure, is the scheme theoretical intersection of $X_{a,b}^\pi$ with the hyperplane section of the lines intersecting $\tau$. 

We first show that $X_{a,b}^{\alpha}$ and $X_{a,b}^{\sigma,\pi}$ are both irreducible for general choices of $\alpha, \pi, \sigma$.

For $X_{a,b}^{\sigma,\pi}$, we notice that it consists of the secants to the rational normal curve $C_r$ intersection of $S_{a,b}$ with $\sigma$, that intersect $\pi$, that has codimension $2$ in $\sigma$. The hyperplanes in $\sigma$ containing $\pi$ cut out a general $g^1_r$ on $C_r$. By generality the monodromy of this $g^1_r$ is the full symmetric group. $X_{a,b}^{\sigma,\pi}$ can be identified with the effective divisors of degree $2$ on $C_r$ that are contained in divisors of the  $g^1_r$. Since the full symmetric group is $2-$transitive, the required irreducibility follows.

As for $X_{a,b}^{\alpha}$ let us consider the projection $f \colon S_{a,b} \rightarrow \pp^{3}$ with center $\alpha$. The image $\Sigma$ of this projection is a surface with ordinary singularities with an irreducible double curve (\cite{F1, F2, MP}). 
$ X_{a,b}^{\alpha}$ can be identified with this double curve, so we have irreducibility. 

Next we go back to the local computation we made in the proof of Lemma \ref {lem:int}, from which we keep the notation. 

We fix $\pi$ and $\tau$ to have equations
$$
\pi)\qquad x_1-x_{b+1}=x_0=x_{b+2}=0, \qquad \tau)\qquad x_1-x_{b+1}=x_{b+3}=0
$$
so that
$$
\sigma)\qquad x_1-x_{b+1}=0, \qquad \alpha)\qquad x_1-x_{b+1}=x_0=x_{b+2}=x_{b+3}=0.
$$
Then $X_{a,b}^\pi$ has equations $\eta_1=1, \eta_{b+2}=0$ and the hyperplane section given by the lines intersecting $\tau$ has equation $\xi_1\eta_{b+3}=\xi_{b+3}(\eta_1-1)$. So the intersection of this hyperplane with $X_{a,b}^\pi$ has equations
$$
\eta_1=1,\quad  \eta_{b+2}=0,\quad  \xi_1\eta_{b+3}=0.
$$
This intersection splits in two parts, one with equations
$$
\eta_1=1,\quad  \eta_{b+2}=0,\quad  \xi_1=0,
$$
that is irreducible and reduced and coincides with  $X_{a,b}^{\sigma,\pi}$, and the other with equations
\begin {equation}\label{ep:lko}
\eta_1=1,\quad  \eta_{b+2}=0,\quad  \eta_{b+3}=0.
\end{equation}
Now, by \eqref {eq:plop} we see that $\eta_{b+3}=\eta_1\xi_{b+2}+\eta_{b+2}^2=\xi_{b+2}$, so that \eqref {ep:lko} is equivalent to
$$
\eta_1=1,\quad \eta_{b+2}=0,\quad  \xi_{b+2}=0
$$
that defines an irreducible and reduced curve, that coincides with $X_{a,b}^{\alpha}$. \end{proof}

We can now prove that:

\begin{thm}\label{thm:deg} One has $\deg (X_{a,b})= 3r^2-8r+6$. 
\end{thm}

\begin{proof} We assume $r\geq 4$, the case $r=3$ is similar, and in fact easier, and can be left to the reader. The case $r=2$ is trivial.

We apply Lemmas \ref {lem:int} and \ref {lem:int2}. First of all we compute $\deg (X_{a,b}^\sigma)$. We note that $X_{a,b}^\sigma$ is nothing but the surface described by the secant lines to a general hyperplane section of $S_{a,b}$ that is a rational normal curve of degree $r$. Hence, by Proposition \ref {prop:rnc}, we have $\deg (X_{a,b}^\sigma)=(r-1)^2$.

Next we compute $\deg (X_{a,b}^{\sigma,\pi})$. This coincides again with the degree of the surface described by the secant lines to a general hyperplane section of $S_{a,b}$. So $\deg (X_{a,b}^{\sigma,\pi})=(r-1)^2$.

Finally, we compute $\deg(X_{a,b}^{\alpha})$. We consider the surface scroll $\Phi\subset \pp^{r+1}$ described by the lines in $X_{a,b}^{\alpha}$.  We claim that 
$$
\deg(X_{a,b}^{\alpha})=\deg(\Phi).
$$
Indeed, $\deg(X_{a,b}^{\alpha})$ equals the number of points in $X_{a,b}^{\alpha}$ corresponding to lines intersecting a general linear subspace $\tau$ of codimension 2, and this is exactly the number of points that $\tau$ has in common with $\Phi$. So we need to compute $\deg(\Phi)$.

The linear subspace $\alpha$ of dimension $r-3$ intersects $\Phi$ along a curve $C$. The curve $C$ can also be interpreted as the scheme theoretical intersection of the secant variety ${\rm Sec}(S_{a,b})$ of $S_{a,b}$ with $\alpha$. 
By the generality of $\alpha$, $C$ is irreducible and reduced, so  $\deg (C)=\deg ({\rm Sec}(S_{a,b}))$.
One has $\deg ({\rm Sec}(S_{a,b}))={{r-2}\choose 2}$. Indeed, if $\beta$ is a general codimension 5 linear space, the number of its intersection points with   ${\rm Sec}(S_{a,b})$ is $\deg ({\rm Sec}(S_{a,b}))$. On the other hand this number is also the number of double points of the  projection of $S_{a,b}$ from $\beta$ to $\pp^4$, that is a general proiection of $S_{a,b}$ to $\pp^4$, and this number is ${{r-2}\choose 2}$ by the double point formula. Thus we have
$$
\deg (C)={{r-2}\choose 2}.
$$
Now take a general hyperplane containing $\alpha$. This intersects $\Phi$ along $C$ and along a certain number $\delta$ of lines of the ruling and
$$
\deg(\Phi)=\delta+\deg (C)=\delta+{{r-2}\choose 2}.
$$
To compute $\delta$, consider the projection from $\alpha$ to $\pp^3$. Then $S_{a,b}$ is projected to a rational scroll $\Sigma$ of degree $r$ in $\pp^3$ with ordinary singularities and  $\delta$ coincides with the degree of the double curve  of $\Sigma$, so that $\delta={{r-1}\choose 2}$. Summing up
$$
\deg(X_{a,b}^{\alpha})=\deg(\Phi)={{r-1}\choose 2}+{{r-2}\choose 2}=(r-2)^2.
$$

In conclusion, putting all the above information together, we get
$$
\deg (X_{a,b})=2(r-1)^2+(r-2)^2
$$
and the assertion follows. \end{proof}

We note that the formula of the Theorem \ref{thm:deg} agrees with the computation in \cite{Cattaneo} of the degree of the image of the secant map for a general surface. However the hypotheses in \cite{Cattaneo} do not apply to our situation.

\subsection{Local structure of \texorpdfstring{$X_{a,b}$}{X{a,b}}} Now we want to study smoothness or singularity of the points of $X_{a,b}$. 

The case $a=b=1$ is trivial since, in this case $X_{1,1}=\mathbb G(1,3)$. So in the rest of this section we will assume $(a,b)\neq (1,1)$ so that $r=a+b>2$.

A preliminary remark is in order. The only lines in $S_{a,b}$ are the ones of the ruling, except for $S_{1,r-1}\cong \F_{r-2}$, in which there is a further line $s$ image of the negative section $E$ of $\F_{r-2}$. This line, in turn, is mapped by the secant map to a point $p_s$ of $X_{1,r-1}$. 

\begin{thm}\label{thm:sing} If $(a,b)\neq (1,1)$, then $X_{a,b}$ is smooth, except for $X_{1,r-1}$, in which case  it has the unique singular point at $p_s$, whose tangent cone is the cone over a smooth threefold linear section of the Segre embedding of  $\pp^2\times \pp^{r-2}$ in $\pp^{3r-4}$, hence it has multiplicity $r\choose 2$. \end{thm}

\begin{proof}  In the proof of Lemma \ref {lem:int} we have proved that given a line of the ruling of $S_{a,b}$, then $X_{a,b}$ is smooth at the point corresponding to that line. 

Let now $\ell$ be a line, not contained in $S_{a,b}$, that is secant (or tangent) to $S_{a,b}$. We want to prove that $X_{a,b}$ is smooth at the point $p_\ell$ corresponding to $\ell$. To see this, we  take the intersection of $X_{a,b}$ with two suitable hyperplanes containing $p_{\ell}$ obtaining a surface and we will show that this surface is smooth at $p_\ell$.

We argue as in the proof of Lemma \ref {lem:int}. Consider a general codimension 3 linear subspace $\pi$ and two distinct codimension 2 linear subspaces $\sigma_1$ and $\sigma_2$ containing $\pi$ and intersecting in one point each, and in different points, $\ell$, and therefore spanning a hyperplane $\sigma$ containing $\ell$. As we saw in Lemma \ref {lem:int} (from which we keep the notation), the intersection of $X_{a,b}$ with the two hyperplanes of lines intersecting $\sigma_1$ and $\sigma_2$ is the reduced union $X_{a,b}^\sigma\cup X_{a,b}^\pi$. The point $p_\ell$ does  not sit in $X_{a,b}^\pi$ but sits in $X_{a,b}^\sigma$, which  is  the surface described by the secant lines to a general hyperplane section of $S_{a,b}$ that is a rational normal curve of degree $r$. Hence, by Proposition \ref {prop:rnc}, $X_{a,b}^\sigma$ is smooth, thus $X_{a,b}^\sigma\cup X_{a,b}^\pi$ is smooth at $p_\ell$ as wanted. 

Let us now turn to the case of $X_{1,r-1}$ and to the point $p_s$. We may assume that $S_{1,r-1}$ is defined by the equations
$$
{\rm rk}
 \left(
\begin{array}{cccccc}
x_0&y_0&y_1&\ldots&y_{r-2}  \\
x_{1}&y_1 &y_2&\ldots&y_{r-1}
\end{array}
\right)<2,
$$
so that the line $s$ is defined by the equations $\{y_i=0\}_{0\leq i\leq r-1}$. An open  neighborhood of $s$ in $\G(1,r+1)$ consists of all  lines joining the points whose homogeneous coordinates are given by the rows of the following matrix
$$
 \left(
\begin{array}{cccccc}
1&0&\xi_0 &\xi_1&\ldots&\xi_{r-1}  \\
0&1&\eta_0&\eta_1&\ldots&\eta_{r-1}
\end{array}
\right)
$$
so that $\xi_0,\ldots, \xi_{r-1}, \eta_0, \ldots, \eta_{r-1}$ are coordinates of a chart $U$ of $\G(1,r+1)$ centered at $s$. A line $t$ parametrized by a point of $U$ has parametric equations of the form
$$
x_0=\lambda, x_1=\mu, y_i=\lambda \xi_i+\mu \eta_i, \quad i=0, \ldots , r-1.
$$
with $[\lambda, \mu]\in \pp^1$. 

The intersection of $t$ with $S_{1,r-1}$ is obtained by solving in $\lambda, \mu$ the system of equations 
$$
{\rm rk}
 \left(
\begin{array}{ccccccccc}
\lambda&\lambda \xi_0+\mu \eta_0&\lambda \xi_1+\mu \eta_1&\ldots&\lambda \xi_{r-2}+\mu \eta_{r-2}   \\
\mu &\lambda \xi_1+\mu \eta_1&\lambda \xi_2+\mu \eta_2&\ldots& \lambda \xi_{r-1}+\mu \eta_{r-1}
\end{array}
\right)<2.
$$
The line $t$ is secant (or tangent) to $S_{a,b}$ if and only if all the degree 2 equations in $\lambda, \mu$ that we obtain in this way are proportional. Let us consider only the equations coming from the minors including the first column. They have the form
$$
\lambda^2\xi_{i+1}+\lambda\mu (\eta_{i+1}-\xi_i)-\mu^2 \eta_i=0, \quad i=0, \ldots, r-2. 
$$
hence the proportionality is given by the equations
\begin{equation}\label{eq: 3-foldSection}
{\rm rk}
 \left(
\begin{array}{ccccccccc}
\xi_1&\xi_2&\ldots&\xi_{r-1}\\
\eta_1-\xi_0&\eta_2-\xi_1&\ldots&\eta_{r-1}-\xi_{r-2}\\
\eta_0&\eta_1&\ldots&\eta_{r-2}\end{array}
\right)<2
\end{equation}
that define the cone over a linear three dimensional section of the Segre embedding of  $\pp^2\times \pp^{r-2}$ in $\pp^{3r-4}$. \end{proof} 

\begin{rem}\label{rem:quod} Suppose $r\geq 3$. Recalling what has been said in \S \ref {sec:1}, we have that the secant map 
$$\gamma_{a,b}:S_{a,b}[2]\longrightarrow X_{a,b}$$  
is:\\
\begin{inparaenum}
\item [$\bullet$] a projective realization of the map $\phi_{b-a,1}: \F_{b-a}[2]\longrightarrow Z_{b-a,1}$, if $a\geq 2$;\\
\item [$\bullet$] a projective realization of the map $\psi_{b-a}: \F_{b-a}[2]\longrightarrow X_{b-a}$, if $a=1$.
\end{inparaenum}

In particular we have that if $2\leq a\leq b$ and $2\leq a'\leq b'$ and $n:=b-a=b'-a'$, then
$X_{a,b}$ is isomorphic to $X_{a',b'}$ and both are isomorphic to $Z_{n,1}$.

This  tells us what is the local nature of the singularity of the variety $Z_{n,2}$, image of the morphism $\phi_{n,2}: \F_n[2]\longrightarrow Z_{n,2}$, at the point $\mathfrak p$ image of the surface $\Ec_n\cong \pp^2$. It suffices to fix $a=1, b=n+1$, so that $r=a+b=n+2$ and we have that $Z_{n,2}$ has at $\mathfrak p$ the same singularity as $X_{1,r-1}$ at the point $\mathfrak p_s$ corresponding to the line $s$ on $S_{1,r-1}$ not belonging to the ruling. So the tangent cone there is the cone over a threefold linear section of the Segre embedding of  $\pp^2\times \pp^{n}$ in $\pp^{3n+2}$, hence it has multiplicity ${n+2} \choose 2$.

If $r=2$, then $a=b=1$, $S_{1,1}\cong \F_0$ is a smooth quadric in $\pp^3$, and the secant map is a projective realization of the map $\psi_0$ appearing in diagram \eqref {eq:diag}, so that $X_0=X_{1,1}$  coincides with $\G(1,3)$, that is a smooth quadric in $\pp^5$. The varieties $Z_{0,1}$ and $Z_{0,2}$ are the blow--ups of $\G(1,3)$ along the two conics $\Gamma$ and $\Gamma'$ that correspond to the two rulings of $S_{1,1}$, and $\F_0[2]$ is the simoultaneous blow--up of these two conics with exceptional divisors $\Fc_0$ and $\Fc'_0$. 
\end{rem}

\section{The case \texorpdfstring{$r=4$}{r=4}}

The case of $S_{2,2}$ has been studied by Fano in \cite {Fa} and, in  recent times, in \cite {AP}. It turns out that $X_{2,2}\subset \pp^{14}$ is a smooth Fano 4--fold of degree 22, of index 2 with canonical curve sections of genus 12. Any smooth hyperplane section of $X_{2,2}$ is a Fano 3--fold of index 1, that has been called in \cite {AP} a \emph{Fano's last Fano} (FlF). 

The surface $S_{1,3}$ is a specialization of $S_{2,2}$ and therefore $X_{1,3}$ is a specialization of $X_{2,2}$. The variety $X_{1,3}\subset \pp^{14}$ has still degree 22 but is no longer smooth, since it has  an isolated singular point of multiplicity 6, with tangent cone the cone over a 3--fold linear section of $\pp^2\times \pp^2$. However the general curve section of $X_{1,3}$ is still smooth and canonical of genus 12. So $X_{1,3}$ is a weak Fano variety, the general threefold section of $X_{1,3}$ is a smooth Fano threefold of index 1, and it is still a FlF.

To better understand in which way the specialization of $X_{2,2}$ to $X_{1,3}$ takes place, it is useful to briefly recall  some of the results in \cite {AP}. 

We have $S_{2,2}\cong \F_0$ and, as we saw in Remark \ref {rem:quod}, $X_{2,2}$ coincides with the variety $Z_{0,1}$, that is the blow--up of the quadric 4--fold $\G(1,3)$ along the conic $\Gamma'$ corresponding to the ruling $|E|$ of $S_{1,1}\cong \F_0$. Note that $\Gamma'$ does not belong to a plane contained in $\G(1,3)$.

It has been proved in \cite {AP} that the projective realization of this is the fact that 
$X_{2,2}\subset \pp^{14}$ is the image of the quadric $\G(1,3)$ via the rational map determined by the linear system $|\mathcal I_{\Gamma',\G(1,3)}(2)|$ of quadric sections of $\G(1,3)$ passing through the conic $\Gamma'$ that does not belong to a plane contained in $\G(1,3)$. As a consequence we have that the general FlF is the blow--up along a conic of a smooth complete intersection of type $(2,2)$ in $\pp^5$. This variety  appears as the
number 16 in the Mori--Mukai list of Fano 3--folds with Picard number 2 (see \cite [Table 2]{MM}).

\begin{rem}\label{rem:qquad}  If $r=2a\geq 4$, then $X_{a,a}$ is isomorphic to $X_{2,2}$ hence it is isomorphic to the blow--up $\tilde \G$ of  $\G(1,3)$ along the conic $\Gamma'$ as above. If $H$ denotes the strict transform on $\tilde \G$ of a general hyperplane section of $\G(1,3)$ and $\mathcal E$ is the exceptional divisor in $\tilde \G$ over the blown--up conic, one sees that the isomorphism of $\tilde \G$ to $X_{a,a}$ is given by the map determined by the linear system $|aH-(a-1)\mathcal E|$. Indeed, one computes
$$
H^4=2, \quad H^3\cdot \Ec=H^2\cdot \Ec^2=0, \quad H\cdot \Ec^3=2, \quad \Ec^4=6
$$
hence 
$$
(aH-(a-1)\mathcal E)^4=12a^2-16a+6=3r^2-8r+6=\deg (X_{a,a}).
$$
To compute $h=\dim (|aH-(a-1)\mathcal E|)$, we first notice that the linear system of hypersurfaces of $\pp^5$ of degree $a$ having multiplicity $a-1$ along a smooth conic, is computed to be $5\frac {(a+1)a} 2$ (the computation can be left to the reader). So $5\frac {(a+1)a} 2-(h+1)$ is the dimension of the linear system of hypersurfaces of degree $a-2$ in $\pp^5$ that have points of multiplicity $a-2$ along a smooth conic, and these are cones with vertex the plane of the conic. Hence
$$
5\frac {(a+1)a} 2-(h+1)=\frac {(a+1)(a-2)}2, 
$$
so that
$$
h=\dim (|aH-(a-1)\mathcal E|)=a(2a+3)=\frac {r(r+3)}2 
$$
that is the embedding dimension of $X_{a,a}$.

It is also interesting to notice that the degree of the exceptional divisor $\Ec$  in $X_{a,a}$ is
$$
(aH-(a-1)\mathcal E)^3\cdot \Ec=6(a-1)^2.
$$
Moreover since $\dim (|aH-a\mathcal E|)=\frac {(a+3)a}2$, we see that the span of $\Ec$ in $X_{a,a}$ has dimension
$$
a(2a+3)-\frac {(a+3)a}2-1=3\frac {(a+1)a}2-1.
$$
In general $\Ec$ is swept out by a 1--dimensional family of $(a-1)$--Veronese surfaces. 
For $a=2$, $\Ec$ is in fact a rational normal scroll threefold of degree 6 in a $\pp^8$. \end{rem}

If we consider $S_{1,3}\cong \F_2$ the above picture changes. To put things in the general perspective we consider the following situation.

Extending the definition of $S_{a,b}$ we can consider $S_{0,r} \subset \pp^{r+1}$ as the cone over the rational curve, image of ${\mathbb F}_r$ via the morphism determined by the linear system $|E+rF|$. Extending the definition of $\gamma_{a,b}$ we obtain a rational map 
$$
\gamma_{0,r}:  \F_{r}[2] \dashrightarrow \G(1,r+1)\subset \pp^{\frac {r(r+3)}2}
$$ 
whose image we denote by $X_{0,r}$. The map is not defined exactly along the surface $\Ec_r \cong \pp^2$ of the pair of points on $E$.

Resolving the indeterminacy we obtain the following commutative diagram, that we are going  to explain. 
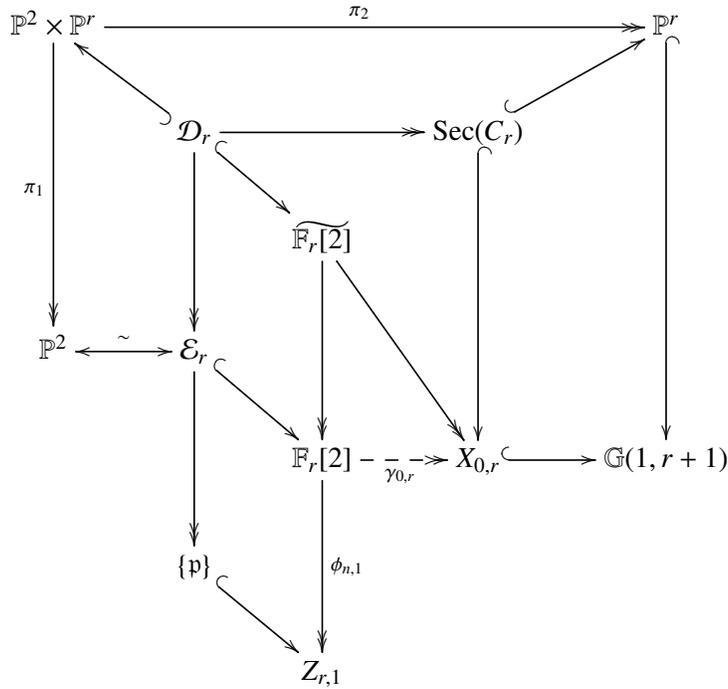
\begin{equation*}
\xymatrix{ 
\pp^2 \times \pp^r \ar@{->>}^{\pi_2}[rrrr]\ar@{->>}_{\pi_1}[ddd]&&&&\pp^r \ar@{^{(}->}[dddd]&&&&&\\
&{\mathcal D}_r \ar@{->>}[dd] \ar@{^{(}->}[lu]\ar@{^{(}->}[rd] \ar@{->>}[rr]&&{\rm Sec} (C_r)\ar@{^{(}->}[ru]\ar@{^{(}->}[ddd] &&&&&&\\
&&\widetilde{\F_r[2]}\ar@{->>}[dd] \ar@{->>}[ddr]&&&&&&&\\
\pp^2\ar@{<->}^{\sim}[r]&\Ec_r  \ar@{^{(}->}[rd]\ar@{->>}[dd] &&&&&&&&\\
&&\F_r[2]\ar@{->>}^{\phi_{n,1}}[dd] \ar@{-->>}_{\gamma_{0,r}}[r] &X_{0,r}\ar@{^{(}->}[r]&\mathbb G(1,r+1)&&&&&\\
&\left\{ {\mathfrak p} \right\} \ar@{^{(}->}[rd]&&&&&&&&\\
&&Z_{r,1}&&&&&&&
}
\end{equation*}

Recall that the map $\phi_{n,1} \colon \F_r[2] \rightarrow Z_{r,1}$ is the contraction of $\Ec_r$ to a singular point $\mathfrak p$.  A neighbourhood of $\mathfrak p$ in $ Z_{r,1}$ is isomorphic to a neighbourhood  $U$ of the only singular point of $X_{1,r+1}$, and $U$ was described in the proof of Theorem \ref{thm:sing} as defined, in a chart of a Grassmannian, by the minors of order $2$ of the $3 \times r$ matrix $A$ analogous to the matrix in \eqref{eq: 3-foldSection} (that was a $3 \times (r-2)$ matrix because we were describing $X_{1,r-1}$). 

Each column of $A$ defines a rational map  $U \dashrightarrow \pp^2$, map that does not depend on the choice of the column, undefined exactly at $\mathfrak p$. The reader can easily check that this map, undefined at $\mathfrak p$, lifts to a morphism on $\F_r[2]$, the morphism $\F_r[2] \rightarrow \pp^1[2] \cong \pp^2$ induced by the ruling $\F_r \rightarrow \pp^1$. 
We blow up $Z_{r,1}$ at $\mathfrak p$. The exceptional divisor over $\mathfrak p$ is (Theorem \ref{thm:sing}) the smooth threefold linear section ${\mathcal D}_r$ of the Segre embedding of $\pp^2 \times \pp^r$ in $\pp^{3r+4}$ defined exactly by minors of order $2$ of the matrix $A$. The columns of $A$ define on ${\mathcal D}_r$  the projection $\pi_1$ on the first factor $\pp^2$ which is then a lift of the analogous map just considered on $\F_r[2] $: this shows that the blow up of $Z_{r,1}$ at $\mathfrak p$ factors through $\F_r[2]$ and in fact coincides with the blow up  $\widetilde{\F_r[2]}$ of  $\F_r[2]$  at $\Ec_r$.
This completes the description of the first three columns of the diagram.

Projecting ${\mathcal D}_r$ to the second factor, $\pp^r$, we obtain a map  that is a $\pp^{3-r}$ bundle if $r=1,2$. If $r\geq 3$, ${\mathcal D}_r$ maps birationally to the secant variety of the rational normal curve. 
In fact $\gamma_{0,r}$, undefined at $\Ec_r$, lifts to a morphism on $\widetilde{\F_r[2]}$ that maps each point of ${\mathcal D}_r$ to a line through the singular point of $S_{0,r}$. The lines through that point form a $\pp^r$ in $\mathbb G(1,r+1)$ and a simple explicit computation in coordinates shows that ${\mathcal D}_r$ maps in it as the second projection $\pi_2$. 

It is worth noticing that the map $\gamma_{0,1}$ is a flip, whereas the map $\gamma_{0,2}$ is a \emph{Mukai flop}, see \cite{WierzbaWisniewski}.

So, when the smooth quadric $S_{1,1}$ flatly degenerates to the singular quadric $S_{0,2}$ (this degeneration can be realized in a linear pencil of quadrics in $\pp^3$), we see that the pair $(\G(1,3), \Gamma')$ (with the conic $\Gamma'$ not contained in a plane of $\G(1,3)$) to be blown up along $\Gamma'$ to get $X_{2,2}$, degenerates to pair $(\G(1,3), \Gamma')$, with $\Gamma'$  contained in a plane of $\G(1,3)$, to be blown up along $\Gamma'$ to get $X_{1,3}$. The singularity of $X_{1,3}$ arises as the contraction to a point of the strict transform of the plane containing $\Gamma'$. 

\begin{rem}\label{rem:eq} We notice that any FlF $Y$ arises as the hyperplane section of $X_{1,3}$. Indeed, as we know, $Y$ can be obtained in the following way. There is $V$, the general complete intersection of two quadrics $Q$ and $Q'$ in $\pp^5$, and there is a smooth conic $\Gamma\subset V$, such that $Y$ is the image in $\pp^{13}$ via the $13$--dimensional linear system $|\mathcal I_{\Gamma,V}(2)|$ of quadric sections of $V$ containing $\Gamma$. Let $\Pi$ be the plane containing $\Gamma$. In the pencil generated by $Q$ and $Q'$, there is a unique quadric $Q_0$ containing $\Pi$. The image of $Q_0$ in $\pp^{14}$ via the $14$--dimensional linear system $|\mathcal I_{\Gamma,Q_0}(2)|$ of quadric sections of $Q_0$ containing $\Gamma$ is an $X_{1,3}$ having $Y$ as a hyperplane section. \end{rem}

\begin{rem}\label{rem:deg} One can consider degenerations of $X_{2,2}$ and of $X_{1,3}$ in the following way. Consider a smooth quadric $Q$ in $\pp^5$ and a singular reduced conic $\Gamma$ contained in $Q$. Then take the image $X$ of $Q$ in $\pp^{14}$ via the $14$--dimensional linear system $|\mathcal I_{\Gamma,Q}(2)|$ of quadric sections of $Q$ containing $\Gamma$. If $\Gamma$ is not contained in a plane of $Q$, $X$ is a degeneration of $X_{2,2}$ that is singular along  a line. If $\Gamma$ is contained in a plane of $Q$, then $X$ is also a degeneration of $X_{1,3}$. 

Similarly, one can consider a singular quadric $Q$ in $\pp^5$ and a smooth  conic $\Gamma$ contained in the smooth locus of $Q$. Then take the image $Z$ of $Q$ in $\pp^{14}$ via the $14$--dimensional linear system $|\mathcal I_{\Gamma,Q}(2)|$ of quadric sections of $Q$ containing $\Gamma$. Then  $Z$ is a singular degeneration of $X_{2,2}$ (and also of $X_{1,3}$, if the plane of $\Gamma$ is contained in $Q$). 

Arguing as in Remark \ref {rem:eq}, we see that a general FlF $Z$ arises as the hyperplane section of such a $X$. 

One can also consider degenerations of FlF to singular threefolds in various ways. We keep the notation of Remark \ref {rem:eq}. One possibility is to keep the complete intersection $V$  of two quadrics $Q$ and $Q'$ in $\pp^5$ smooth, but to take the conic $\Gamma$, to be blown up to get the FlF $Y$, singular but reduced. In this case $Y$ has a simple double point. 

Another possibility, suggested to us by Ivan Cheltsov, is to take $V$ singular and the conic $\Gamma$ still smooth. For example, we can take $V$ as  the image of $\pp^3$ via the linear system of quadrics passing through four points $p_1,\ldots, p_4\in \pp^3$ in general position. If we blow--up $\pp^3$ at $p_1,\ldots, p_4$, with exceptional divisors $E_1,\ldots, E_4$, the above map becomes a morphism on this blow--up $\widetilde { \pp^3}$, that maps $E_1,\ldots, E_4$ to four planes and contracts the strict transforms of the six lines pairwise joining $p_1,\ldots, p_4$ to six simple double points.  There are many conics on $V$, for instance the images of general lines of $\pp^3$. If we blow--up one of them we get a weak FlF double at six points. 
\end{rem}

\section{A GIT analysis}

As we noticed in Remark \ref {rem:eq}, any FlF is a hyperplane section of $X_{1,3}\subset \pp^{14}$. 
Let $G$ be the automorphism group $G$ of $S_{1,3}={\mathbb F}_2$. 
The group $G$ has dimension 7 and it is isomorphic to the automorphism group of a quadric cone in $\pp^3$. 
This is an extension of ${\rm PGL}(2,\C)$ with the 4--dimensional normal subgroup $G_0$ of the projective transformation of $\pp^3$ that fix a point $x$ and also map every line of $\pp^3$ through $x$ to itself. 
The group $G$ acts as a group of projective trasformations of $X_{1,3}$. 
Recall that $X_{1,3}$ has a unique singular point $p$ and contains a rational normal quartic curve $\Gamma$ that is the image of the lines of the ruling of $S_{1,3}$. 
Of course $G$ fixes $p$ and maps $\Gamma$ to itself inducing on it the action of ${\rm PGL}(2,\C)$. 
Moreover $G$ acts in a natural way on the dual space $\mathcal H:=(\pp^{14})^\vee$ (and on the vector space $H^0(X_{1,3}, \Oc_{X_{1,3}}(1))$) by acting on the hyperplane sections of $X_{1,3}$ that, if smooth, are FlFs. 
 We want to understand the $G$--semistable elements of $\mathcal H$. A partial answer to this question is the following:
 
 \begin{prop}\label{prop:sstab} 
 Let $H$ be a hyperplane section of $X_{1,3}$ not containing the curve $\Gamma$ and cutting out on $\Gamma$ a divisor not containing a point with multiplicity 3. Then $H$ is $G$--semistable.
 \end{prop}
 
 \begin{proof} We argue by contradiction. Suppose that $H$ is not semistable. Let $s\in H^0(X_{1,3}, \Oc_{X_{1,3}}(1))$ be a non--zero section, determined up to a constant, that vanishes on $H$. Then there is a sequence $\{g_n\}_{n\in \mathbb N}$ of elements of $G$ such that $\lim_{n}g_n.s=0$. But then, if $\sigma$ is the restriction of $s$ to $\Gamma$ we also have $\lim_{n}g_n.\sigma=0$ and this is a contradiction because the divisor cut out by $H$ on $\Gamma$ is ${\rm PGL}(2,\C)$--semistable (see \cite [Prop. 4.1] {M}). 
 \end{proof}
 
 The conclusion is that $\mathcal H^{\rm ss}$ is non--empty, and therefore there exists the categorical quotient $\Mm=\mathcal H^{\rm ss}\varparallel  G$, that has dimension $7$. Note that the Kuranishi family of a FlF has also dimension $7$ (see \cite{fanography, belmansfatighentitanturri}). 
Hence locally we can consider $\Mm$ as a finite cover of the Kuranishi family. It follows that $G$ is the component of the identity of the group of projective automorphisms of $X_{1,3}$ (we believe that in fact $G$ is equal to the group of projective transformations of $X_{1,3}$ but we have not been able to prove it so far). Moreover  we can consider somehow $\Mm$  as a moduli space of FlFs. 
However some crucial questions remains open. For example:\\
 \begin{inparaenum}
 \item [(i)] are all smooth hyperplane sections of $X_{1,3}$ $G$ (semi)stable?\\
 \item [(ii)] are two smooth hyperplane sections of $X_{1,3}$ isomorphic if and only if they are $G$--isomorphic?\\
\item [(iii)] are FlFs $K$--stable? 
 \end{inparaenum}
 
An affirmative answer to question (iii) would provide a true moduli space for FlFs. We will not deal with these questions here.

\end{document}